\UseRawInputEncoding
\documentclass[12pt]{amsart}
\usepackage{}

\usepackage{amsmath}
\usepackage{amsfonts}
\usepackage{amssymb}
\usepackage[all]{xy}           

\usepackage{bbding}
\usepackage{txfonts}
\usepackage{amscd}

\usepackage[shortlabels]{enumitem}
\usepackage{ifpdf}
\ifpdf
  \usepackage[colorlinks,final,backref=page,hyperindex]{hyperref}
\else
  \usepackage[colorlinks,final,backref=page,hyperindex,hypertex]{hyperref}
\fi
\usepackage{tikz}
\usepackage[active]{srcltx}
\usepackage{cancel}
\makeatletter

\newtheorem{thm}{Theorem}[section]
\newtheorem{lem}[thm]{Lemma}

\newtheorem{pro}[thm]{Proposition}
\newtheorem{ex}[thm]{Example}
\newtheorem{rmk}[thm]{Remark}
\newtheorem{defi}[thm]{Definition}

\newcommand {\emptycomment}[1]{}

\newcommand{\lon }{\,\rightarrow\,}
\newcommand{\be }{\begin{equation}}
\newcommand{\ee }{\end{equation}}

\newcommand{\g}{\mathfrak g}
\newcommand{\h}{\mathfrak h}

\newcommand{\huaB}{\mathcal{B}}

\newcommand{\huaP}{\mathcal{P}}
\newcommand{\huaC}{{\mathcal{C}}}

\newcommand{\huaH}{\mathcal{H}}

\newcommand{\huaZ}{\mathcal{Z}}

\newcommand{\frkH}{\mathfrak H}

\newcommand{\Id}{{\rm{Id}}}

\newcommand{\br}[1]{   [ \cdot,    \cdot  ]   }

\newcommand{\Hom}{\mathrm{Hom}}

\newcommand{\gl}{\mathfrak {gl}}

\newcommand{\ad}{\mathrm{ad}}

\newcommand{\Li}{\mathsf{3Lie}}

\begin{document}

\title[Cohomology and the controlling algebra of crossed homomorphisms on $3$-Lie algebras]{Cohomology and the controlling algebra of crossed homomorphisms on $3$-Lie algebras}

\author{Shuai Hou}
\address{Department of Mathematics, Jilin University, Changchun 130012, Jilin, China}
\email{houshuai19@mails.jlu.edu.cn}

\author{Meiyan Hu}
\address{Department of Mathematics, Jilin University, Changchun 130012, Jilin, China}
\email{hmy21@mails.jlu.edu.cn}

\author{Lina Song}
\address{Department of Mathematics, Jilin University, Changchun 130012, Jilin, China}
\email{songln@jlu.edu.cn}

\author{Yanqiu Zhou}
\address{School of Science, Guangxi University of Science and Technology, Liuzhou 545006, China}
\email{zhouyanqiunihao@163.com}


\begin{abstract}
In this paper, first we give the notion of a crossed homomorphism on a $3$-Lie algebra with respect to an action on another $3$-Lie algebra, and characterize it using a homomorphism from a Lie algebra to the semidirect product Lie algebra. We also establish the relationship between crossed homomorphisms and relative Rota-Baxter operators of weight $1$ on 3-Lie algebras.
Next we construct a cohomology theory for a crossed homomorphism on $3$-Lie algebras and classify infinitesimal deformations of crossed homomorphisms using the second cohomology group. Finally, using the higher derived brackets, we construct an $L_\infty$-algebra whose Maurer-Cartan elements are crossed homomorphisms. Consequently, we obtain the twisted $L_\infty$-algebra that controls deformations of a given crossed homomorphism on $3$-Lie algebras.
\end{abstract}

\renewcommand{\thefootnote}{}

\maketitle

\tableofcontents

\allowdisplaybreaks


\section{Introduction}
The notion of $3$-Lie algebras and more generally, $n$-Lie
algebras (also called Filippov algebras) was introduced in \cite{Filippov}. See the review article \cite{review,Makhlouf} for more details. The $n$-Lie algebra is the algebraic structure corresponding to Nambu mechanics \cite{Nambu}. In recent years, $3$-Lie algebras have been widely studied and applied in the fields of mathematics and physics, especially in string theory and  M2-branes \cite{Bagger,HHM,Gustavsson}. For example, metric $3$-Lie algebras play a significant role in the basic model of Bagger-Lambert-Gustavsson theory \cite{dMFM,Medeiros}, the supersymmetric Yang-Mills theory can be studied by a special structure of $3$-Lie algebras \cite{Gomis}, and in \cite{Basu}, Basu and Harvey suggested to replace the Lie algebra appearing in the Nahm equation by a $3$-Lie algebra for the lifted Nahm equations.

The notion of a crossed homomorphism of on Lie algebras was introduced by Lue in \cite{Lue}. A crossed homomorphism is also called a relative difference operator or differential operator of weight $1$ with respect to the adjoint representation \cite{GUOK,GSZ,Liu-Guo}. Crossed homomorphisms are related to post-Lie algebras and can be used to study the integration of post-Lie algebras  \cite{Mencattini}. In \cite{PSTZ}, using the crossed homomorphisms on Lie algebras, they studied the relationship between the category of weak representations of Lie-Rinehart algebras and the monoidal category of representations of Lie algebras of Cartan type. They also introduced the cohomology theory of crossed homomorphisms on Lie algebras and studied linear deformations of crossed homomorphisms. In \cite{JSH}, the authors studied the controlling algebra of relative difference Lie algebras and defined the cohomology of difference Lie algebras with coefficients in arbitrary representations. Crossed homomorphisms on Hopf algebras and Cartier-Kostant-Milnor-Moore theorem for difference Hopf algebras were studied in \cite{GLT}.

The research on the deformation theory of algebraic structures began with the seminal work of Gerstenhaber  for associative algebras \cite{Ge}. Next, Nijenhuis and Richardson extended the study of deformation theory to Lie algebra
\cite{NR}. In \cite{FO,Makhlouf}, the deformation problem of $n$-Lie algebras and $3$-Lie algebras were studied respectively. See the review \cite{GlST} for more details. Recently, the deformations of certain operators, e.g. morphisms, relative Rota-Baxter operators on $3$-Lie algebras have been  deeply studied \cite{Arfa,THS}. Actually,  an invertible linear map is a differential operator if and only its inverse is a (relative) Rota-Baxter operator on $3$-Lie algebras \cite{BaiRGuo,BGS-3-Bialgebras}.

The purpose of this  paper is to study crossed homomorphisms on 3-Lie algebras, with particular interests in the cohomology and deformation theories. The crossed homomorphisms introduced in this paper are closely related to relative Rota-Baxter operators of weight $1$ on 3-Lie algebras introduced in \cite{HouSZ}. More precisely, the inverse of an invertible crossed homomorphism is a relative Rota-Baxter operators  of weight $1$, which generalizes the classical relations between crossed homomorphisms and relative Rota-Baxter operators of weight 1 on Lie algebras, and thus justifies its correctness. A crossed homomorphism gives rise to a new representation, and the corresponding cohomology of 3-Lie algebras is taken to be the cohomology of the crossed homomorphism. As expected, the second cohomology group classifies infinitesimal deformations of the crossed homomorphism. Furthermore, we use
Voronov's higher derived brackets to construct an $L_\infty$-algebra whose Maurer-Cartan elements are crossed homomorphisms. Consequently, we obtain the $L_\infty$-algebra governing deformations of a crossed homomorphism. Note that in the Lie algebra context, it is a differential graded Lie algebra governing deformations of a crossed homomorphism on Lie algebras. While for 3-Lie algebras, it is indeed an $L_\infty$-algebra with nontrivial $l_3$ governing deformations of a crossed homomorphism, which is totally different from the case of Lie algebras.

 The paper is organized as follows. In Section \ref{sec:two}, we introduce the notion of crossed homomorphisms on $3$-Lie algebras and illustrate the relationship between crossed homomorphisms and relative Rota-Baxter operator of weight $1$.  In Section \ref{sec:cohomology},  we establish the cohomology theory of crossed homomorphisms on $3$-Lie algebras.  We use the cohomology theory of crossed homomorphisms that we established to classify infinitesimal deformations of crossed homomorphisms. In Section \ref{sec:four}, we
construct an $L_{\infty}$-algebra whose Maurer-Cartan elements are precisely crossed homomorphisms on $3$-Lie algebras. We also using Getzler's twisted $L_{\infty}$-algebra theory to  characterize  deformations of crossed homomorphisms on $3$-Lie algebras.

\vspace{2mm}

In this paper, we work over an algebraically closed filed $\mathbb K$ of characteristic $0$.


\vspace{2mm}
\noindent
{\bf Acknowledgements.} This research is supported by NSFC (12001226).  

\section{Crossed homomorphisms on $3$-Lie algebras}\label{sec:two}
In this section, we introduce the notion of crossed homomorphisms on $3$-Lie algebras, and find that there is a close relationship between crossed homomorphisms and relative Rota-Baxter operators of weight $1$ on $3$-Lie algebras.

\begin{defi}{\rm (\cite{Filippov})}\label{defi:3Lie}
A {\bf 3-Lie algebra}
is a vector space $\g$ together with a skew-symmetric linear map $[\cdot,\cdot,\cdot]_{\g}:\wedge^{3}\g\rightarrow \g$, such that for $ x_{i}\in \g, 1\leq i\leq 5$, the following {\bf Fundamental Identity} holds:
\begin{eqnarray}
\nonumber\qquad &&[x_1,x_2,[x_3,x_4, x_5]_{\g}]_{\g}\\
&=&[[x_1,x_2, x_3]_{\g},x_4,x_5]_{\g}+[x_3,[x_1,x_2, x_4]_{\g},x_5]_{\g}+[x_3,x_4,[x_1,x_2, x_5]_{\g}]_{\g}.
 \label{eq:jacobi1}
\end{eqnarray}
\end{defi}
For $x_{1},x_{2}\in \g$, define $\ad_{x_1,x_2}\in \gl(\g)$ by
\begin{eqnarray*}\label{eq2}
\ad_{x_{1},x_{2}}x:=[x_{1},x_{2},x]_{\g},\quad \forall x\in \g.
\end{eqnarray*}
Then $\ad_{x_{1},x_{2}}$ is a derivation, i.e.
$$\ad_{x_{1},x_{2}}[x_{3},x_{4},x_{5}]_{\g}=[\ad_{x_{1},x_{2}}x_{3},x_{4},x_{5}]_{\g}+
[x_{3},\ad_{x_{1},x_{2}}x_{4},x_{5}]_{\g}+[x_{3},x_{4},\ad_{x_{1},x_{2}}x_{5}]_{\g}.$$
\begin{defi}{\rm (\cite{KA})}
A {\bf representation} of a $3$-Lie algebra $(\g,[\cdot,\cdot,\cdot]_{\g})$ on a vector space $V$ is a linear
map: $\rho:\wedge^{2}\g\rightarrow \gl(V)$, such that for all $x_{1}, x_{2}, x_{3}, x_{4}\in \g,$ the following equalities hold:
\begin{eqnarray}
~\label{representation-1}\rho(x_{1},x_{2})\rho(x_{3},x_{4})&=&\rho([x_{1},x_{2},x_{3}]_{\g},x_{4})+
\rho(x_{3},[x_{1},x_{2},x_{4}]_{\g})+\rho(x_{3},x_{4})\rho(x_{1},x_{2}),\\
~\label{representation-2}\rho(x_{1},[x_{2},x_{3},x_{4}]_{\g})&=&\rho(x_{3},x_{4})\rho(x_{1},x_{2})-\rho(x_{2},x_{4})\rho(x_{1},x_{3})
+\rho(x_{2},x_{3})\rho(x_{1},x_{4}).
\end{eqnarray}
\end{defi}

Let $(\g,[\cdot,\cdot,\cdot]_{\g})$ be a $3$-Lie algebra.  The linear map $\ad:\wedge^2\g\rightarrow\gl(\g)$ defines a representation
of the $3$-Lie algebra $\g$ on itself, which is called the {\bf adjoint representation} of $\g.$

\begin{defi}{\rm (\cite{Filippov})}
Let $(\g,[\cdot,\cdot,\cdot]_{\g})$ be a $3$-Lie algebra. Then the subalgebra $[\g,\g,\g]_{\g}$ is called the {\bf derived algebra} of $\g$, and denoted by ${\g}^1$. 
The subspace $$\huaC(\g)=\{x\in \g~|~[x,y,z]_{\g}=0,~ \forall y,z\in\g\}$$   is called the {\bf center} of $\g$.
\end{defi}

\begin{defi}\cite{HouSZ}
Let $(\g,[\cdot,\cdot,\cdot]_{\g})$ and $(\h,[\cdot,\cdot,\cdot]_{\h})$  be two $3$-Lie algebras.
Let $\rho: \wedge^2\g\rightarrow \gl(\h)$ be a representation of the $3$-Lie algebra $\g$ on the vector space $\h$.
If
for all $x,y\in \g, u,v,w\in \h,$
\begin{eqnarray}
\label{eq:action-1}{}\rho(x,y)u\in \huaC(\h),\\
\label{eq:action-2}{}\rho(x,y)[u,v,w]_{\h}=0,
\end{eqnarray}
then $\rho$ is called  an {\bf action} of  $\g$ on  $\h.$ We denote an action by $(\h;\rho).$
\end{defi}

\begin{pro}\cite{HouSZ}\label{lem:semi}
  Let $\rho: \wedge^2\g\rightarrow \gl(\h)$ be an action of a $3$-Lie algebra $(\g,[\cdot,\cdot,\cdot]_{\g})$ on a $3$-Lie algebra $(\h,[\cdot,\cdot,\cdot]_{\h}).$ There is a $3$-Lie algebra structure on $\g\oplus \h$, defined by
   \begin{eqnarray}
{}[x+u,y+v,z+w]_{\rho}=[x,y,z]_{\g}+\rho(x,y)w+\rho(y,z)u+\rho(z,x)v+[u,v,w]_\h,
\end{eqnarray}
  for all $x,y,z\in\g,~u,v,w\in\h.$ This $3$-Lie algebra is called the  semidirect product of the $3$-Lie algebra $\g$ and the $3$-Lie algebra $\h$ with respect to the action $\rho$, and denoted by $\g\ltimes _\rho\h.$
\end{pro}

Next we give the notion of crossed homomorphisms on $3$-Lie algebras.
\begin{defi}\label{crossed-homo}
  Let $\rho:\wedge^2\g\to\gl(\h)$ be an action of a $3$-Lie algebra $(\g,[\cdot,\cdot,\cdot]_\g)$  on a $3$-Lie algebra  $(\h,[\cdot,\cdot,\cdot]_\h)$.  A linear map $H:\g\to \h$ is called a {\bf crossed homomorphism with respect to the action $\rho$} if
\begin{eqnarray}\label{eq:crossed-homo}
\qquad H[x,y,z]_\g=\rho(x,y)(Hz)+\rho(y,z)(Hx)+\rho(z,x)(Hy)+[Hx,Hy,Hz]_\h,\quad \forall x, y, z\in \g.
\end{eqnarray}
\end{defi}

\begin{rmk}
  If the action $\rho$ of $\g$ on $\h$ is zero, then any crossed homomorphism from $\g$ to $\h$ is
nothing but a $3$-Lie algebra homomorphism. If~$\h$ is commutative, then any crossed homomorphism
from $\g$ to $\h$ is simply a derivation from $\g$ to $\h$ with respect to the representation $(\h; \rho)$.
\end{rmk}

In the Lie algebra context, crossed homomorphisms play important roles in the study of representations of Lie algebras of Cartan type. An essential ingredient in the whole theory is that a crossed homomorphism $H:\g\to\h$ induces a homomorphism from the Lie algebra $\g$ to the semidirect product Lie algebra $\g\ltimes \h$ (\cite[Theorem 2.7]{PSTZ}). Now for crossed homomorphisms on 3-Lie algebras, we still have this characterization.

\begin{thm}
 Let $\rho: \wedge^2\g\rightarrow \gl(\h)$ be an action of a $3$-Lie algebra $(\g,[\cdot,\cdot,\cdot]_{\g})$ on a $3$-Lie algebra $(\h,[\cdot,\cdot,\cdot]_{\h}).$ Then a linear map $H:\g\rightarrow\h$ is a crossed homomorphism from $\g$ to $\h$ if and only if the map
 $\phi_{H}:\g\rightarrow\g\ltimes _\rho\h$ defined by
 \begin{eqnarray}
   \phi_{H}(x):=(x,Hx),\quad \forall x\in\g,
 \end{eqnarray}
 is a $3$-Lie algebra homomorphism.
\end{thm}
\begin{proof}
For all $x,y,z\in\g,$ we have
 \begin{eqnarray*}
  \phi_{H}[x,y,z]_{\g}&=&([x,y,z]_{\g},H[x,y,z]_{\g});\\
  {}[\phi_{H}(x),\phi_{H}(y),\phi_{H}(z)]_{\rho}
   &=&([x,y,z]_{\g},\rho(x,y)(Hz)+\rho(y,z)(Hx)+\rho(z,x)(Hy)+[Hx,Hy,Hz]_{\h}).
 \end{eqnarray*}
 Thus, $\phi_{H}[x,y,z]_{\g}=[\phi_{H}(x),\phi_{H}(y),\phi_{H}(z)]_{\rho}$ if and only if $H$ is a crossed homomorphism from $\g$ to $\h$ with respect to the action $\rho.$
\end{proof}

\begin{ex}
Let $(\g,[\cdot,\cdot,\cdot]_{\g})$ be a $4$-dimensional $3$-Lie algebra with a basis $\{e_1,e_2,e_3,e_4\}$ and the nonzero multiplication is given by $$[e_2,e_3,e_4]_{\g}=e_1.$$
The center of $\g$ is the subspace generated by $\{e_1\}.$ It is obvious that
  the adjoint representation $\ad:\wedge^2\g\rightarrow\gl(\g)$
is an action of $\g$ on itself.
For a matrix $\left(\begin{array}{cccc}
a_{11}&a_{12}&a_{13}&a_{14}\\
a_{21}&a_{22}&a_{23}&a_{24}\\
a_{31}&a_{32}&a_{33}&a_{34}\\
a_{41}&a_{42}&a_{43}&a_{44}
\end{array}\right),$
define
\begin{align*}
  He_1=&a_{11}e_1+a_{21}e_2+a_{31}e_3+a_{41}e_4,\quad He_2=a_{12}e_1+a_{22}e_2+a_{32}e_3+a_{42}e_4,\\
  He_3=&a_{13}e_1+a_{23}e_2+a_{33}e_3+a_{43}e_4,\quad He_4=a_{14}e_1+a_{24}e_2+a_{34}e_3+a_{44}e_4.
\end{align*}
$H$ is a crossed homomorphism from $\g$ to $\g$ with respect to the action $\ad$ if and only if
\begin{eqnarray*}
H[e_i,e_j,e_k]_{\g}=[He_i,e_j,e_k]_{\g}+[e_i,He_j,e_k]_{\g}+[e_i,e_j,He_k]_{\g}+[He_i,He_j,He_k]_{\g},\quad i,j,k=1,2,3,4.
\end{eqnarray*}
By straightforward computations, we deduce that $H$ is a crossed homomorphism if and only if
\begin{eqnarray*}
\left\{\begin{array}{rcl}
{}a_{11}&=&a_{22}+a_{33}+a_{44}+a_{23}a_{34}a_{42}+a_{24}a_{32}a_{43}+a_{22}a_{33}a_{44}-a_{24}a_{33}a_{42}-a_{22}a_{34}a_{43}-a_{23}a_{32}a_{44},\\
{}a_{21}&=&a_{31}=a_{41}=0.
\end{array}\right.
\end{eqnarray*}
In particular,
$\begin{cases}
H(e_1)=0,\\
H(e_2)=e_2,\\
H(e_3)=e_3,\\
H(e_4)=-e_4,
\end{cases}$is a crossed homomorphism from $\g$ to $\g$ with respect to the adjoint action $\ad$.
\end{ex}

\begin{defi}
Let $H$ and $H'$ be two crossed homomorphisms from a $3$-Lie algebra $(\g,[\cdot,\cdot,\cdot]_\g)$ to a $3$-Lie algebra $(\h,[\cdot,\cdot,\cdot]_\h)$ with respect to an action $\rho$. A {\bf homomorphism} from $H$ to $H'$ consists of $3$-Lie algebra homomorphisms $\psi_\g: \g\lon\g$ and   $\psi_\h: \h\lon\h$ such that
\begin{eqnarray}
 \label{condition-1}\psi_\h\circ H&=&H'\circ\psi_\g,\\
  \label{condition-2}\psi_\h(\rho(x,y)u)&=&\rho(\psi_\g(x),\psi_\g(y))(\psi_\h(u)),\quad \forall x,y\in \g, u\in \h.
\end{eqnarray}
\emptycomment{
i.e. we have the following commutative diagram
\[\xymatrix{
 \mathfrak g \ar[d]_{H} \ar[r]^{\psi_{\g}}
                & \mathfrak g \ar[d]^{H'}  \\
  \mathfrak h \ar[r]^{\psi_{\h}}
                & \mathfrak h             }\]}
In particular, if both $\psi_\g$ and $\psi_\h$ are invertible, $(\psi_\g, \psi_\h)$ is called an {\bf isomorphism} from $H$ to $H'$.
\end{defi}
\begin{lem}
Let $H:\g\rightarrow\h$ be a crossed homomorphism from $\g$ to $\h$ with respect to an action $\rho$.
Let $\psi_\g: \g\lon\g$ and  $\psi_\h: \h\lon\h$ be $3$-Lie algebra isomorphisms such that \eqref{condition-2} holds.
Then $\psi^{-1}_{\h}\circ H\circ\psi_{\g}$ is a crossed homomorphism from $\g$ to $\h$ with respect to the action $\rho$.
\end{lem}
\begin{proof}
For all $x,y,z\in\g$, we have
\begin{eqnarray*}
  &&(\psi^{-1}_{\h}\circ H\circ\psi_{\g})[x,y,z]_{\g}\\
  &=&\psi^{-1}_{\h}\Big(\rho(\psi_{\g}(x),\psi_{\g}(y))H(\psi_{\g}(z))+\rho(\psi_{\g}(y),\psi_{\g}(z))H(\psi_{\g}(x))
  +\rho(\psi_{\g}(z),\psi_{\g}(x))H(\psi_{\g}(y))\\
  &&+[H(\psi_{\g}(x)),H(\psi_{\g}(y)),H(\psi_{\g}(z))]_{\h}\Big)\\
  &=&\rho(x,y)(\psi^{-1}_{\h}\circ H\circ\psi_{\g}(z))+\rho(y,z)(\psi^{-1}_{\h}\circ H\circ\psi_{\g}(x))+\rho(z,x)(\psi^{-1}_{\h}\circ H\circ\psi_{\g}(y))\\
  &&+[\psi^{-1}_{\h}\circ H\circ\psi_{\g}(x),\psi^{-1}_{\h}\circ H\circ\psi_{\g}(y),\psi^{-1}_{\h}\circ H\circ\psi_{\g}(z)]_{\h},
\end{eqnarray*}
which implies that $\psi^{-1}_{\h}\circ H\circ\psi_{\g}$ is a crossed homomorphism.
\end{proof}

At the end of this section, we establish the relationship between crossed homomorphisms and relative Rota-Baxter operators of weight $1$ on $3$-Lie algebras.

Recall from \cite{HouSZ} that
a linear map $T: \h\rightarrow\g$ is called a {\bf
  relative Rota-Baxter operator}  of weight $\lambda\in \mathbb K$ from a $3$-Lie algebra $\h$ to a $3$-Lie algebra $\g$ with respect to an action $\rho$  if
\begin{eqnarray} \label{eq:rRB}
  [Tu,Tv,Tw]_\g
=T\Big(\rho(Tu,Tv)w+\rho(Tv,Tw)u+\rho(Tw,Tu)v+\lambda[u,v,w]_\h\Big),
\end{eqnarray}
for all $u, v, w\in\h.$

\begin{pro}
  Let $\rho:\wedge^2\g\to\gl(\h)$ be an action of a $3$-Lie algebra $(\g,[\cdot,\cdot,\cdot]_\g)$  on a $3$-Lie algebra  $(\h,[\cdot,\cdot,\cdot]_\h)$.  An invertible linear map $H:\g\rightarrow\h$ is a crossed homomorphism from the $3$-Lie algebra $\g$ to the $3$-Lie algebra $\h$ with respect to the action $\rho$ if and only if $H^{-1}$ is a relative Rota-Baxter operator of weight $1$ from the $3$-Lie
algebra $\h$ to the $3$-Lie algebra $\g$ with respect to the action $\rho$.
\end{pro}
\begin{proof}
 If an invertible linear map $H:\g\rightarrow\h$ is a crossed homomorphism, then for $u_1, u_2, u_3\in
 \h,$ by \eqref{eq:crossed-homo}, we have
 \begin{align*}
   &[H^{-1}(u_1),H^{-1}(u_2),H^{-1}(u_3)]_{\g}\\
   =&H^{-1}(H[H^{-1}(u_1),H^{-1}(u_2),H^{-1}(u_3)]_{\g})\\
   =&H^{-1}\Big(\rho(H^{-1}(u_1),H^{-1}(u_2))(u_3)+\rho(H^{-1}(u_2),H^{-1}(u_3))(u_1)+\rho(H^{-1}(u_3),H^{-1}(u_1))(u_2)+[u_1,u_2,u_3]_{\h}\Big).
 \end{align*}
Therefore, $H^{-1}$ is a relative Rota-Baxter operator of weight $1$.

Conversely, if $H^{-1}$ be a relative Rota-Baxter operator of weight $1$. For all $x_1, x_2, x_3\in\g,$ assume $x_i=H^{-1}(u_i),1\leq i\leq3,$ for $u_i\in \h.$ By \eqref{eq:rRB}, we have
\begin{align*}
 &H[x_1,x_2,x_3]_{\g}\\
 =&H[H^{-1}(u_1),H^{-1}(u_2),H^{-1}(u_3)]_{\g}\\
 =&H(H^{-1}(\rho(H^{-1}(u_1),H^{-1}(u_2))(u_3)+\rho(H^{-1}(u_2),H^{-1}(u_3))(u_1)+\rho(H^{-1}(u_3),H^{-1}(u_1))(u_2)+[u_1,u_2,u_3]_{\h}))\\
 =&\rho(x_1,x_2)H(x_3)+\rho(x_2,x_3)H(x_1)+\rho(x_3,x_1)H(x_2)+[Hx_1,Hx_2,Hx_3]_{\h}.
\end{align*}
So $H$ is a crossed homomorphism.
\end{proof}

\section{Cohomologies of crossed homomorphisms on $3$-Lie algebras}\label{sec:cohomology}
In this section, we define the cohomology of crossed homomorphisms on $3$-Lie algebras, and we use the second cohomology group to study infinitesimal deformations of crossed homomorphisms.
\subsection{Cohomologies of crossed homomorphisms} First, we recall the cohomologies theory of $3$-Lie algebras.

Let $(V;\rho)$ be a representation of a $3$-Lie algebra  $(\g,[\cdot,\cdot,\cdot]_{\g})$. Denote by
$$\mathfrak C_{\Li}^{n}(\g;V):=\Hom (\underbrace{\wedge^{2} \g\otimes \cdots\otimes \wedge^{2}\g}_{(n-1)}\wedge \g,V),\quad(n\geq 1),$$
which is the space of $n$-cochains.
The coboundary operator ${\rm d}:\mathfrak C_{\Li}^{n}(\g;V)\rightarrow \mathfrak C_{\Li}^{n+1}(\g;V)$ is defined by
\begin{eqnarray*}&&
({\rm d}f)(\mathfrak{X}_1,\cdots,\mathfrak{X}_n,x_{n+1})\\
&=&\sum_{1\leq j<k\leq n}(-1)^{j} f(\mathfrak{X}_1,\cdots,\hat{\mathfrak{X}_{j}},\cdots,\mathfrak{X}_{k-1},
[x_j,y_j,x_k]_{\g}\wedge y_k\\&&+x_k\wedge[x_j,y_j,y_k]_{\g},
\mathfrak{X}_{k+1},\cdots,\mathfrak{X}_{n},x_{n+1})\\&&
+\sum_{j=1}^{n}(-1)^{j}f(\mathfrak{X}_1,\cdots,\hat{\mathfrak{X}_{j}},\cdots,\mathfrak{X}_{n},
[x_j,y_j,x_{n+1}]_{\g})\\&&
+\sum_{j=1}^{n}(-1)^{j+1}\rho(x_j,y_j)f(\mathfrak{X}_1,\cdots,\hat{\mathfrak{X}_{j}},
\cdots,\mathfrak{X}_{n},x_{n+1})\\&&
+(-1)^{n+1}\Big(\rho(y_n,x_{n+1})f(\mathfrak{X}_1,\cdots,\mathfrak{X}_{n-1},x_n)+\rho(x_{n+1},x_n)f(\mathfrak{X}_1,\cdots,\mathfrak{X}_{n-1},y_n)\Big),
\end{eqnarray*}
for all$~\mathfrak{X}_{i}=x_{i}\wedge y_{i}\in \wedge^{2}\g,~i=1,2,\cdots,n~and~x_{n+1}\in \g.$ It was proved in \cite{Casas,Takhtajan1} that ${\rm d}\circ{\rm d}=0.$ Thus, $(\oplus_{n=1}^{+\infty}\mathfrak C_{\Li}^{n}(\g;V),{\rm d})$ is a cochain complex.

\begin{defi}
The {\bf cohomology} of the $3$-Lie algebra $\g$ with coefficients in $V$ is the cohomology of the cochain complex $(\oplus_{n=1}^{+\infty} \mathfrak C_{\Li}^{n}(\g;V),{\rm d})$. Denote by $\huaZ_{\Li}^{n}(\g;V)$ and $\huaB_{\Li}^{n}(\g;V)$
the set of $n$-cocycles and the set of $n$-coboundaries, respectively. The $n$-th cohomology group is defined by
\begin{eqnarray*}
\huaH_{\Li}^{n}(\g;V)=\huaZ_{\Li}^{n}(\g;V)/\huaB_{\Li}^{n}(\g;V).
\end{eqnarray*}
\end{defi}

\begin{lem}
  Let $H$ be a crossed homomorphism from a $3$-Lie algebra $(\g,[\cdot,\cdot,\cdot]_{\g})$ to a $3$-Lie algebra $(\h,[\cdot,\cdot,\cdot]_{\h})$ with respect to an action $\rho.$ Define $\rho_{H}:\wedge^2\g\rightarrow\gl(\h)$ by
  \begin{eqnarray}\label{crossed-representation}
    \rho_{H}(x,y)u:=\rho(x,y)u+[Hx,Hy,u]_{\h}, \quad \forall x,y\in\g,u\in \h.
      \end{eqnarray}
  Then $\rho_{H}$ is a representation of $\g$ on $\h$.
\end{lem}
\begin{proof}
By a direct calculation using \eqref{eq:jacobi1}-\eqref{eq:crossed-homo},
for all $x_i\in \g,1\leq i\leq 4, u\in \h,$ we have
\begin{eqnarray*}
&&\Big(\rho_{H}(x_1,x_2)\rho_{H}(x_3,x_4)-\rho_{H}(x_3,x_4)\rho_{H}(x_1,x_2)\\
&&\quad -\rho_{H}([x_1,x_2,x_3]_{\g},x_4)+\rho_{H}([x_1,x_2,x_4]_{\g},x_3)\Big)(u)\\
&=&\rho(x_1,x_2)\rho(x_3,x_4)u+[Hx_1,Hx_2,\rho(x_3,x_4)u]_{\h}+\rho(x_1,x_2)[Hx_3,Hx_4,u]_{\h}\\
&&+[Hx_1,Hx_2,[Hx_3,Hx_4,u]_{\h}]_{\h}-\rho(x_3,x_4)\rho(x_1,x_2)u-[Hx_3,Hx_4,\rho(x_1,x_2)u]_{\h}\\
&&-\rho(x_3,x_4)[Hx_1,Hx_2,u]_{\h}-[Hx_3,Hx_4,[Hx_1,Hx_2,u]_{\h}]_{\h}-\rho([x_1,x_2,x_3]_{\g},x_4)u\\
&&-[H[x_1,x_2,x_3]_{\g},Hx_4,u]_{\h}+\rho([x_1,x_2,x_4]_{\g},x_3)u+[H[x_1,x_2,x_4]_{\g},Hx_3,u]_{\h}\\
&=&0,
\end{eqnarray*}
and
\begin{eqnarray*}
&&\Big(\rho_{H}([x_1,x_2,x_3]_{\g},x_4)-\rho_{H}(x_1,x_2)\rho_{H}(x_3,x_4)\\
&&\quad -\rho_{H}(x_2,x_3)\rho_{H}(x_1,x_4)-\rho_{H}(x_3,x_1)\rho_{H}(x_2,x_4)\Big)(u)\\
&=&\rho([x_1,x_2,x_3]_{\g},x_4)u+[H[x_1,x_2,x_3]_{\h},Hx_4,u]_{\h}-\rho(x_1,x_2)\rho(x_3,x_4)u\\
&&-[Hx_1,Hx_2,\rho(x_3,x_4)u]_{\h}-\rho(x_1,x_2)[Hx_3,Hx_4,u]_{\h}-[Hx_1,Hx_2,[Hx_3,Hx_4,u]_{\h}]_{\h}\\
&&-\rho(x_2,x_3)\rho(x_1,x_4)u-[Hx_2,Hx_3,\rho(x_1,x_4)u]_{\h}-\rho(x_2,x_3)[Hx_1,Hx_4,u]_{\h}\\
&&-[Hx_2,Hx_3,[Hx_1,Hx_4,u]_{\h}]_{\h}-\rho(x_3,x_1)\rho(x_2,x_4)u-[Hx_3,Hx_1,\rho(x_2,x_4)u]_{\h}\\
&&-\rho(x_3,x_1)[Hx_2,Hx_4,u]_{\h}-[Hx_3,Hx_1,[Hx_2,Hx_4,u]_{\h}]_{\h}\\
&=&0.
\end{eqnarray*}
Therefore, we deduce that $(\h;\rho_{H})$ is a representation of the $3$-Lie algebra $(\g,[\cdot,\cdot,\cdot]_{\g})$.
\end{proof}

Let ${\rm d}_{\rho_{H}}:\mathfrak C_{\Li}^{n}(\g;\h)\rightarrow \mathfrak C_{\Li}^{n+1}(\g;\h),(n\geq1)$ be the corresponding coboundary
operator of the $3$-Lie algebra $(\g,[\cdot,\cdot,\cdot]_{\g})$ with coefficients in the representation $(\h;\rho_{H})$.
More precisely, for all $f\in \Hom (\underbrace{\wedge^{2} \g\otimes \cdots\otimes \wedge^{2}\g}_{(n-1)}\wedge \g,\h)$, $\mathfrak{X}_i=x_i\wedge y_i\in \wedge^2\g,~ i=1,2,\cdots,n$ and $x_{n+1}\in \g,$ we have
\begin{eqnarray*}
&&({\rm d}_{\rho_{H}}f)(\mathfrak{X}_1,\cdots,\mathfrak{X}_n,x_{n+1})\\
&=&\sum_{1\leq i<k\leq n}(-1)^{i} f(\mathfrak{X}_1,\cdots,\hat{\mathfrak{X}_{i}},\cdots,\mathfrak{X}_{k-1},
[x_i,y_i,x_k]_{\g}\wedge y_k\\&&+x_k\wedge[x_i,y_i,y_k]_{\g},
\mathfrak{X}_{k+1},\cdots,\mathfrak{X}_{n},x_{n+1})\\&&
+\sum_{i=1}^{n}(-1)^{i}f(\mathfrak{X}_1,\cdots,\hat{\mathfrak{X}_{i}},\cdots,\mathfrak{X}_{n},
[x_i,y_i,x_{n+1}]_{\g})\\&&
+\sum_{i=1}^{n}(-1)^{i+1}\rho(x_i,y_i)f(\mathfrak{X}_1,\cdots,\hat{\mathfrak{X}_{i}},
\cdots,\mathfrak{X}_{n},x_{n+1})\\&&
+\sum_{i=1}^{n}(-1)^{i+1}[Hx_i,Hy_i,f(\mathfrak{X}_1,\cdots,\hat{\mathfrak{X}_{i}},
\cdots,\mathfrak{X}_{n},x_{n+1})]_{\h}\\&&
+(-1)^{n+1}\Big(\rho(y_n,x_{n+1})f(\mathfrak{X}_1,\cdots,\mathfrak{X}_{n-1},x_n)+[Hy_n,Hx_{n+1},f(\mathfrak{X}_1,\cdots,\mathfrak{X}_{n-1},x_n)]_{\h} \Big)\\&&
+(-1)^{n+1}\Big(\rho(x_{n+1},x_n)f(\mathfrak{X}_1,\cdots,\mathfrak{X}_{n-1},y_n)+[Hx_{n+1},Hx_n,f(\mathfrak{X}_1,\cdots,\mathfrak{X}_{n-1},y_n)]_{\h}\Big).
\end{eqnarray*}

It is obvious that $f\in \Hom(\g;\h)$ is closed if and only if
\begin{eqnarray*}
f([x_1,x_2,x_3]_{\g})&=&\rho(x_1,x_2)f(x_3)+[Hx_1,Hx_2,f(x_3)]_{\h}+\rho(x_2,x_3)f(x_1)\\
&&+[Hx_2,Hx_3,f(x_1)]_{\h}+\rho(x_3,x_1)f(x_2)+[Hx_3,Hx_1,f(x_2)]_{\h},\quad \forall x_1,x_2,x_3\in \g.
\end{eqnarray*}

Define $\delta:\wedge^2\g\rightarrow\Hom(\g,\h)$ by
\begin{eqnarray*}
\delta(\mathfrak{X})z=\rho(y,z)H(x)+\rho(z,x)H(y)+[Hx,Hy,Hz]_{\h}, \quad \forall\mathfrak{X}=x\wedge y\in\wedge^2\g, z\in \g.
\end{eqnarray*}
\begin{pro}
 Let $H$ be a crossed homomorphism from a $3$-Lie algebra $(\g,[\cdot,\cdot,\cdot]_{\g})$ to a $3$-Lie algebra $(\h,[\cdot,\cdot,\cdot]_{\h})$ with respect to an action $\rho.$  Then $\delta(\mathfrak{X})$ is a $1$-cocycle of the $3$-Lie algebra $(\g,[\cdot,\cdot,\cdot]_{\g})$ with coefficients in $(\h;\rho_{H}).$
\end{pro}
\begin{proof}
For all $x_1,x_2,x_3\in \g,$ by \eqref{eq:jacobi1}-\eqref{eq:crossed-homo}, we have
\begin{eqnarray*}
&&({\rm d}_{\rho_{H}}\delta(\mathfrak{X}))(x_1,x_2,x_3)\\
&=&\rho(x_1,x_2)\delta(\mathfrak{X})(x_3)+\rho(x_2,x_3)\delta(\mathfrak{X})(x_1)+\rho(x_3,x_1)\delta(\mathfrak{X})(x_2)-\delta(\mathfrak{X})([x_1,x_2,x_3]_{\g})\\
&&+[Hx_1,Hx_2,\delta(\mathfrak{X})x_3]_{\h}+[\delta(\mathfrak{X})x_1,Hx_2,Hx_3]_{\h}+[Hx_1,\delta(\mathfrak{X})x_2,Hx_3]_{\h}\\
&=&\rho(x_1,x_2)\Big([Hx,Hy,Hx_3]_{\h}+\rho(y,x_3)(Hx)+\rho(x_3,x)(Hy)\Big)\\
&&+\rho(x_2,x_3)\Big([Hx,Hy,Hx_1]_{\h}+\rho(y,x_1)(Hx)+\rho(x_1,x)(Hy)\Big)\\
&&+\rho(x_3,x_1)\Big([Hx,Hy,Hx_2]_{\h}+\rho(y,x_2)(Hx)+\rho(x_2,x)(Hy)\Big)\\
&&+[Hx_1,Hx_2,[Hx,Hy,Hx_3]_{\h}]_{\h}+[Hx_1,Hx_2,\rho(y,x_3)(Hx)]_{\h}+[Hx_1,Hx_2,\rho(x_3,x)(Hy)]_{\h}\\
&&+[Hx_2,Hx_3,[Hx,Hy,Hx_1]_{\h}]_{\h}+[Hx_2,Hx_3,\rho(y,x_1)(Hx)]_{\h}+[Hx_2,Hx_3,\rho(x_1,x)(Hy)]_{\h}\\
&&+[Hx_3,Hx_1,[Hx,Hy,Hx_2]_{\h}]_{\h}+[Hx_3,Hx_1,\rho(y,x_2)(Hx)]_{\h}+[Hx_3,Hx_1,\rho(x_2,x)(Hy)]_{\h}\\
&&-[Hx,Hy,H[x_1,x_2,x_3]_{\g}]_{\h}-\rho(y,[x_1,x_2,x_3]_{\g})(Hx)-\rho([x_1,x_2,x_3]_{\g},x)(Hy)\\
&=&0.
\end{eqnarray*}
Thus, we deduce that ${\rm d}_{\rho_{H}}\delta(\mathfrak{X})=0.$ The proof is finished.
\end{proof}
We now give the cohomology of crossed homomorphisms on $3$-Lie algebras.

Let $H$ be a crossed homomorphism from a $3$-Lie algebra $(\g,[\cdot,\cdot,\cdot]_{\g})$ to a $3$-Lie algebra $(\h,[\cdot,\cdot,\cdot]_{\h})$ with respect to an action $\rho.$ Define the set of $n$-cochains by
\begin{eqnarray}\label{crossed-cochain}
\mathfrak C_{H}^{n}(\g;\h)=
\left\{\begin{array}{rcl}
{}\mathfrak C_{\Li}^{n-1}(\g;\h),\quad n\geq 2,\\
{}\g\wedge\g,\quad n=1.
\end{array}\right.
\end{eqnarray}

Define ${\partial}:\mathfrak C_{H}^{n}(\g;\h)\rightarrow \mathfrak C_{H}^{n+1}(\g;\h)$ by
\begin{eqnarray}\label{crossed-cohomology}
{\partial}=
\left\{\begin{array}{rcl}
{}{\rm d}_{\rho_{H}},\quad n\geq 2,\\
{}\delta,\quad n=1.
\end{array}\right.
\end{eqnarray}
Then $(\mathop{\oplus}\limits_{n=1}^{\infty} \mathfrak C_{H}^{n}(\g;\h),\partial)$ is a cochain complex. Denote the set of $n$-cocycles by $\huaZ^n_H(\g;\h),$ the set of $n$-coboundaries by $\huaB^n_H(\g;\h)$ and $n$-th cohomology group by
\begin{eqnarray}\label{crossed-cohomology-group}
\huaH^n_H(\g;\h)=\huaZ_H^n(\g;\h)/\huaB_H^n(\g;\h),\quad n\geq1.
\end{eqnarray}

\begin{defi}
The cohomology of the cochain complex $(\mathop{\oplus}\limits_{n=1}^{\infty} \mathfrak C_{H}^{n}(\g;\h),\partial)$  is taken to be the {\bf cohomology for the crossed homomorphism $H$}.
\end{defi}

At the end of this subsection, we show that certain homomorphisms between crossed homomorphisms induce homomorphisms between the corresponding cohomology groups.
Let $H$ and $H'$ be two crossed homomorphisms from a $3$-Lie algebra $(\g,[\cdot,\cdot,\cdot]_{\g})$ to a $3$-Lie algebra $(\h,[\cdot,\cdot,\cdot]_{\h})$ with respect to an action $\rho.$ Let $(\psi_\g, \psi_\h)$ be a homomorphism from $H$ to $H'$ in which $\psi_\g$ is invertible. Define a map $p:\mathfrak C_{H}^{n}(\g;\h)\rightarrow \mathfrak C_{H'}^{n}(\g;\h)$ by
\begin{eqnarray*}
p(\omega)(\mathfrak{X}_1,\cdots,\mathfrak{X}_{n-2},x_{n-1})=\psi_{\h}\Bigg(\omega\Big(\psi^{-1}_{\g}(x_1)\wedge\psi^{-1}_{\g}(y_1),\cdots,\psi^{-1}_{\g}(x_{n-2})\wedge\psi^{-1}_{\g}(y_{n-2}),\psi^{-1}_{\g}(x_{n-1})\Big)\Bigg),
\end{eqnarray*}
for all $\omega\in \mathfrak C_{H}^{n}(\g;\h),\mathfrak{X}_i=x_i\wedge y_i\in \wedge^2\g,~ i=1,2,\cdots,n-2$ and $x_{n-1}\in \g.$

\begin{thm}
With above notations, $p$ is a cochain map from the cochain complex $(\mathop{\oplus}\limits_{n=2}^{\infty} \mathfrak C_{H}^{n}(\g;\h),{\rm d}_{\rho_{H}})$
to the cochain complex $(\mathop{\oplus}\limits_{n=2}^{\infty} \mathfrak C_{H'}^{n}(\g;\h),{\rm d}_{\rho_{H'}})$. Consequently, it induces a homomorphism $p_*$ from the
cohomology group $\huaH^{n}_H(\g;\h)$ to $\huaH^{n}_{H'}(\g;\h)$.
\end{thm}

\begin{proof}
For all $\omega\in \mathfrak C_{H}^{n}(\g;\h)$, by \eqref{condition-1}-\eqref{condition-2} and \eqref{crossed-cochain}-\eqref{crossed-cohomology-group}, we have
\begin{align*}
&{\rm d}_{\rho_{H'}}(p(\omega))(\mathfrak{X}_1,\cdots,\mathfrak{X}_{n-1},x_{n})\\
=&\sum_{1\leq i<k\leq n-1}(-1)^{i} p(\omega)(\mathfrak{X}_1,\cdots,\hat{\mathfrak{X}_{i}},\cdots,\mathfrak{X}_{k-1},
[x_i,y_i,x_k]_{\g}\wedge y_k\\
&+x_k\wedge[x_i,y_i,y_k]_{\g},
\mathfrak{X}_{k+1},\cdots,\mathfrak{X}_{n-1},x_{n})\\
&+\sum_{i=1}^{n-1}(-1)^{i}p(\omega)(\mathfrak{X}_1,\cdots,\hat{\mathfrak{X}_{i}},\cdots,\mathfrak{X}_{n},
[x_i,y_i,x_{n}]_{\g})\\&
+\sum_{i=1}^{n-1}(-1)^{i+1}\rho(x_i,y_i)p(\omega)(\mathfrak{X}_1,\cdots,\hat{\mathfrak{X}_{i}},
\cdots,\mathfrak{X}_{n-1},x_{n})\\&
+\sum_{i=1}^{n-1}(-1)^{i+1}[H'x_i,H'y_i,p(\omega)(\mathfrak{X}_1,\cdots,\hat{\mathfrak{X}_{i}},
\cdots,\mathfrak{X}_{n-1},x_{n})]_{\h}\\&
+(-1)^{n}\Big(\rho(y_{n-1},x_{n})p(\omega)(\mathfrak{X}_1,\cdots,\mathfrak{X}_{n-2},x_{n-1})+[H'y_{n-1},H'x_{n},p(\omega)(\mathfrak{X}_1,\cdots,\mathfrak{X}_{n-2},x_{n-1})]_{\h} \Big)\\&
+(-1)^{n}\Big(\rho(x_{n},x_{n-1})p(\omega)(\mathfrak{X}_1,\cdots,\mathfrak{X}_{n-2},y_{n-1})+[H'x_{n},H'x_{n-1},p(\omega)(\mathfrak{X}_1,\cdots,\mathfrak{X}_{n-2},y_{n-1})]_{\h}\Big)\\
=&\sum_{1\leq i<k\leq n-1}(-1)^i\psi_{\h}\Big(\omega(\psi^{-1}_{\g}(x_1)\wedge\psi^{-1}_{\g}(y_1),\cdots,\hat{\psi^{-1}_{\g}(x_i)}\wedge\hat{\psi^{-1}_{\g}(y_i)},\cdots,\psi^{-1}_{\g}(x_{k-1})\wedge\psi^{-1}_{\g}(y_{k-1}),\\
&\psi^{-1}_{\g}([x_i,y_i,x_k]_{\g})\wedge \psi^{-1}_{\g}(y_k)+\psi^{-1}_{\g}(x_k)\wedge\psi^{-1}_{\g}([x_i,y_i,y_k]_{\g}),\cdots,\psi^{-1}_{\g}(x_{n-1})\wedge\psi^{-1}_{\g}(y_{n-1}),\psi^{-1}_{\g}(x_{n}))\Big)\\
&+\sum_{i=1}^{n-1}(-1)^{i}\psi_{\h}\Big(\omega(\psi^{-1}_{\g}(x_1)\wedge\psi^{-1}_{\g}(y_1),\cdots,
\hat{\psi^{-1}_{\g}(x_i)}\wedge\hat{\psi^{-1}_{\g}(y_i)},\cdots,\\&
\psi^{-1}_{\g}(x_{n-1})\wedge\psi^{-1}_{\g}(y_{n-1}),\psi^{-1}_{\g}([x_i,y_i,x_{n}]_{\g}))\Big)+\sum_{i=1}^{n-1}(-1)^{i+1}\rho(x_i,y_i)\psi_{\h}\Big(\omega(\psi^{-1}_{\g}(x_1)\wedge\psi^{-1}_{\g}(y_1),\cdots,\\
&\hat{\psi^{-1}_{\g}(x_i)}\wedge\hat{\psi^{-1}_{\g}(y_i)},\cdots,\psi^{-1}_{\g}(x_{n-1})\wedge\psi^{-1}_{\g}(y_{n-1}),\psi^{-1}_{\g}(x_{n}))\Big)\\&
+\sum_{i=1}^{n-1}(-1)^{i+1}[H'x_i,H'y_i,\psi_{\h}\Big(\omega(\psi^{-1}_{\g}(x_1)\wedge\psi^{-1}_{\g}(y_1),\cdots,\\
&\hat{\psi^{-1}_{\g}(x_i)}\wedge\hat{\psi^{-1}_{\g}(y_i)},\cdots,\psi^{-1}_{\g}(x_{n-1})\wedge\psi^{-1}_{\g}(y_{n-1}),\psi^{-1}_{\g}(x_{n}))\Big)]_{\h}\\&
+(-1)^{n}\Big(\rho(y_{n-1},x_{n})\psi_{\h}\Big(\omega(\psi^{-1}_{\g}(x_1)\wedge\psi^{-1}_{\g}(y_1),\cdots,\psi^{-1}_{\g}(x_{n-2})\wedge\psi^{-1}_{\g}(y_{n-2}),\psi^{-1}_{\g}(x_{n-1}))\Big)\\
&+[H'y_{n-1},H'x_{n},\psi_{\h}\Big(\omega(\psi^{-1}_{\g}(x_1)\wedge\psi^{-1}_{\g}(y_1),\cdots,\psi^{-1}_{\g}(x_{n-2})\wedge\psi^{-1}_{\g}(y_{n-2}),\psi^{-1}_{\g}(x_{n-1}))\Big)]_{\h} \Big)\\&
+(-1)^{n}\Big(\rho(x_{n},x_{n-1})\psi_{\h}\Big(\omega(\psi^{-1}_{\g}(x_1)\wedge\psi^{-1}_{\g}(y_1),\cdots,\psi^{-1}_{\g}(x_{n-2})\wedge\psi^{-1}_{\g}(y_{n-2}),\psi^{-1}_{\g}(y_{n-1}))\Big)\\
&+[H'x_{n},H'x_{n-1},\psi_{\h}\Big((\omega(\psi^{-1}_{\g}(x_1)\wedge\psi^{-1}_{\g}(y_1),\cdots,\psi^{-1}_{\g}(x_{n-2})\wedge\psi^{-1}_{\g}(y_{n-2}),\psi^{-1}_{\g}(y_{n-1}))\Big)]_{\h}\Big)\\
=&\sum_{1\leq i<k\leq n-1}(-1)^i\psi_{\h}\Big(\omega(\psi^{-1}_{\g}(x_1)\wedge\psi^{-1}_{\g}(y_1),\cdots,\hat{\psi^{-1}_{\g}(x_i)}\wedge\hat{\psi^{-1}_{\g}(y_i)},\cdots,\psi^{-1}_{\g}(x_{k-1})\wedge\psi^{-1}_{\g}(y_{k-1}),\\
&([\psi^{-1}_{\g}(x_i),\psi^{-1}_{\g}(y_i),\psi^{-1}_{\g}(x_k)]_{\g})\wedge \psi^{-1}_{\g}(y_k)+\psi^{-1}_{\g}(x_k)\wedge([\psi^{-1}_{\g}(x_i),\psi^{-1}_{\g}(y_i),\psi^{-1}_{\g}(y_k)]_{\g}),\cdots,\\&
\psi^{-1}_{\g}(x_{n-1})\wedge\psi^{-1}_{\g}(y_{n-1}),\psi^{-1}_{\g}(x_{n}))\Big)
+\sum_{i=1}^{n-1}(-1)^{i}\psi_{\h}\Big(\omega(\psi^{-1}_{\g}(x_1)\wedge\psi^{-1}_{\g}(y_1),\cdots,
\hat{\psi^{-1}_{\g}(x_i)}\wedge\hat{\psi^{-1}_{\g}(y_i)},\cdots,\\&
\psi^{-1}_{\g}(x_{n-1})\wedge\psi^{-1}_{\g}(y_{n-1}),([\psi^{-1}_{\g}(x_i),\psi^{-1}_{\g}(y_i),\psi^{-1}_{\g}(x_{n})]_{\g}))\Big)+\sum_{i=1}^{n-1}(-1)^{i+1}\psi_{\h}\rho(\psi^{-1}_{\g}(x_i),\psi^{-1}_{\g}(y_i))\\
&\Big(\omega(\psi^{-1}_{\g}(x_1)\wedge\psi^{-1}_{\g}(y_1),\cdots,\hat{\psi^{-1}_{\g}(x_i)}\wedge\hat{\psi^{-1}_{\g}(y_i)},\cdots,\psi^{-1}_{\g}(x_{n-1})\wedge\psi^{-1}_{\g}(y_{n-1}),\psi^{-1}_{\g}(x_{n}))\Big)\\&
+\sum_{i=1}^{n-1}(-1)^{i+1}\psi_{\h}[H(\psi^{-1}_{\g}(x_i)),H(\psi^{-1}_{\g}(y_i)),\omega(\psi^{-1}_{\g}(x_1)\wedge\psi^{-1}_{\g}(y_1),\cdots,\\
&\hat{\psi^{-1}_{\g}(x_i)}\wedge\hat{\psi^{-1}_{\g}(y_i)},\cdots,\psi^{-1}_{\g}(x_{n-1})\wedge\psi^{-1}_{\g}(y_{n-1}),\psi^{-1}_{\g}(x_{n}))]_{\h}\\&
+(-1)^{n}\psi_{\h}\Big(\rho(\psi^{-1}_{\g}(y_{n-1}),\psi^{-1}_{\g}(x_{n}))\omega(\psi^{-1}_{\g}(x_1)\wedge\psi^{-1}_{\g}(y_1),\cdots,\psi^{-1}_{\g}(x_{n-2})\wedge\psi^{-1}_{\g}(y_{n-2}),\psi^{-1}_{\g}(x_{n-1}))\\
&+[H\psi^{-1}_{\g}(y_{n-1}),H\psi^{-1}_{\g}(x_{n}),\omega(\psi^{-1}_{\g}(x_1)\wedge\psi^{-1}_{\g}(y_1),\cdots,\psi^{-1}_{\g}(x_{n-2})\wedge\psi^{-1}_{\g}(y_{n-2}),\psi^{-1}_{\g}(x_{n-1})) ]_{\h} \Big)\\&
+(-1)^{n}\psi_{\h}\Big(\rho(\psi^{-1}_{\g}(x_{n}),\psi^{-1}_{\g}(x_{n-1}))\omega(\psi^{-1}_{\g}(x_1)\wedge\psi^{-1}_{\g}(y_1),\cdots,\psi^{-1}_{\g}(x_{n-2})\wedge\psi^{-1}_{\g}(y_{n-2}),\psi^{-1}_{\g}(y_{n-1}))\\
&+[H\psi^{-1}_{\g}(x_{n}),H\psi^{-1}_{\g}(x_{n-1}),\omega(\psi^{-1}_{\g}(x_1)\wedge\psi^{-1}_{\g}(y_1),\cdots,\psi^{-1}_{\g}(x_{n-2})\wedge\psi^{-1}_{\g}(y_{n-2}),\psi^{-1}_{\g}(y_{n-1}))]_{\h}\Big)\\
=&\psi_{\h}({\rm d}_{\rho_{H}}\omega)\Big(\psi^{-1}_{\g}(x_1)\wedge\psi^{-1}_{\g}(y_1),\cdots,\psi^{-1}_{\g}(x_{n-1})\wedge\psi^{-1}_{\g}(y_{n-1}),\psi^{-1}_{\g}(x_{n})\Big)\\
=&p({\rm d}_{\rho_{H}}\omega)(\mathfrak{X}_1,\cdots,\mathfrak{X}_{n-1},x_{n}),
\end{align*}
where $\mathfrak{X}_i=x_i\wedge y_i\in \wedge^2\g,~ i=1,2,\cdots,n-1$ and $x_{n}\in \g.$
Thus $p$ is a cochain map, and induces a homomorphism $p_*$ from the cohomology
group $\huaH^{n}_H(\g;\h)$ to $\huaH^{n}_{H'}(\g;\h)$.
\end{proof}
\subsection{Infinitesimal deformations of crossed homomorphisms}\label{sec:defor}

In this section, we use the established cohomology theory to characterize  infinitesimal deformations of crossed homomorphisms on 3-Lie algebras.

Let $(\g,[\cdot,\cdot,\cdot]_{\g})$ be a $3$-Lie algebra over $\mathbb K$ and $\mathbb K[t]$ be the polynomial ring in one variable $t.$
Then $\mathbb K[t]/(t^2)\otimes_{\mathbb K}\g$ is an $\mathbb K[t]/(t^2)$-module. Moreover, $\mathbb K[t]/(t^2)\otimes_{\mathbb K}\g$ is a $3$-Lie algebra over $\mathbb K[t]/(t^2)$, where the $3$-Lie algebra structure is defined by
\begin{eqnarray*}
[f_1(t)\otimes_{\mathbb K} x_1,f_2(t)\otimes_{\mathbb K} x_2,f_3(t)\otimes_{\mathbb K} x_3]= f_1(t)f_2(t) f_3(t)\otimes_{\mathbb K}[x_1,x_2,x_3]_{\g},
\end{eqnarray*}
for $f_{i}(t)\in \mathbb K[t]/(t^2),1\leq i\leq 3,x_1,x_2,x_3\in \g.$

In the sequel, all the vector spaces are finite dimensional vector spaces over $\mathbb K$ and we denote $f(t)\otimes_{\mathbb K} x$ by $f(t)x,$ where $f(t)\in \mathbb K[t]/(t^2).$

\begin{defi}
Let $H:\g\rightarrow \h$ be a crossed homomorphism from a $3$-Lie algebra $(\g,[\cdot,\cdot,\cdot]_\g)$ to a $3$-Lie algebra $(\h,[\cdot,\cdot,\cdot]_\h)$ with respect to an action $\rho$. Let
$\frkH:\g\rightarrow\h$ be a linear map. If $H_t=H+t\frkH$ is still a crossed homomorphism
 modulo $t^2$, they we say that $\frkH$ generates an {\bf infinitesimal deformation} of the crossed homomorphism $H$.
\end{defi}

Since $H_t=H+t\frkH $ is a crossed homomorphism, for any $x,y,z\in \h,$
we have
\begin{eqnarray}
\label{equivalent-1}\qquad\frkH[x,y,z]_{\g}&=&\rho(x,y)\frkH(z)+\rho(y,z)\frkH(x)+\rho(z,x)\frkH(y)\\
\nonumber                                             &&+[\frkH(x),Hy,Hz]_{\h}+[Hx,\frkH(y),Hz]_{\h}+[H(x),Hy,\frkH(z)]_{\h};
\end{eqnarray}

 Note that \eqref{equivalent-1} means that $\frkH$ is a $2$-cocycle of the crossed homomorphism  $H$. Hence, $\frkH$ defines a cohomology class in $\huaH^2_{H}(\g;\h)$.

\begin{defi}
Let $H$ be a crossed homomorphism from a $3$-Lie algebra $(\g,[\cdot,\cdot,\cdot]_\g)$ to a $3$-Lie algebra $(\h,[\cdot,\cdot,\cdot]_\h)$ with respect to an action $\rho$. Two one-parameter infinitesimal deformations $H^1_{t}=H+t\frkH_{1}$ and  $H^2_{t}=H+t\frkH_{2}$ are said to be {\bf equivalent} if there exists $\mathfrak{X}\in\g\wedge\g$ such that $(\Id_{\g}+t\ad_{\mathfrak{X}},\Id_{\h}+t\rho(\mathfrak{X}))$ is a homomorphism modulo $t^2$ from $H^1_{t}$ to $H^2_{t}$.
In particular, an infinitesimal deformation $H_{t}=H+t\frkH_{1}$ of a crossed homomorphism  $H$ is said to be {\bf trivial} if there exists $\mathfrak{X}\in \g\wedge\g$ such that $(\Id_{\g}+t\ad_{\mathfrak{X}},\Id_{\h}+t\rho(\mathfrak{X}))$ is a homomorphism modulo $t^2$  from $H_{t}$ to $H.$
\end{defi}

Let $(\Id_{\g}+t\ad_{\mathfrak{X}},\Id_{\h}+t\rho(\mathfrak{X}))$ be a homomorphism modulo $t^2$ from $H^1_{t}$ to $H^2_{t}.$
By \eqref{condition-1}, we get,
\begin{equation*}
(\Id_\h+t\rho(\mathfrak{X}))(H+t\frkH_1)(z)=(H+t\frkH_2)(\Id_{\g}+t\ad_{\mathfrak{X}})(z),\quad \forall\mathfrak{X}=x\wedge y\in\wedge^2\g, z\in \g.
\end{equation*}
which implies
\begin{eqnarray}\label{Nijenhuis-element-4}
\frkH_1(z)-\frkH_2(z)&=&\rho(y,z)(Hx)+\rho(z,x)(Hy)+[Hx,Hy,Hz]_{\h}.
\end{eqnarray}
\emptycomment{
\begin{eqnarray}\label{Nijenhuis-element-4}
\left\{\begin{array}{rcl}
{}\frkH_1(z)-\frkH_2(z)&=&\rho(y,z)(Hx)+\rho(z,x)(Hy)+[Hx,Hy,Hz]_{\h},\\
{}\frkH_2[x,y,z]_{\g}&=&\rho(x,y)\frkH_1(x).
\end{array}\right.
\end{eqnarray}}

Now we are ready to give the main result in this section.
\begin{thm}
Let $H$ be a crossed homomorphism from a $3$-Lie algebra $(\g,[\cdot,\cdot,\cdot]_\g)$ to a $3$-Lie algebra $(\h,[\cdot,\cdot,\cdot]_\h)$ with respect to an action $\rho$.
If two one-parameter infinitesimal deformations $H^1_{t}=H+t\frkH_{1}$ and $H^2_{t}=H+t\frkH_{2}$ are equivalent, then $\frkH_{1}$ and $\frkH_{2}$
 are in the same cohomology class in $\huaH^2_{H}(\g;\h)$.
\end{thm}
\begin{proof}
 It is easy to see from the condition \eqref{Nijenhuis-element-4} that
\begin{eqnarray*}
 \frkH_1(z)&=& \frkH_2(z)+(\partial\mathfrak{X})(z),\quad \forall z\in \g,
\end{eqnarray*}
which implies that $\frkH_1$ and $\frkH_2$ are in the same cohomology class.
\end{proof}

\section{Maurer-Cartan characterization of crossed homomorphisms on $3$-Lie algebras}\label{sec:four}
In this section, we construct a suitable $L_{\infty}$-algebra,
which characterize crossed homomorphisms on $3$-Lie algebras  as Maurer-Cartan elements. Then we construct a twisted $L_{\infty}$-algebra
that controls deformations of crossed homomorphisms.
\begin{defi}
An {\em  $L_\infty$-algebra} is a $\mathbb Z$-graded vector space $\g=\oplus_{k\in\mathbb Z}\g^k$ equipped with a collection $(k\ge 1)$ of linear maps $l_k:\otimes^k\g\lon\g$ of degree $1$ with the property that, for any homogeneous elements $x_1,\cdots,x_n\in \g$, we have
\begin{itemize}\item[\rm(i)]
{\em (graded symmetry)} for every $\sigma\in\mathbb S_{n}$,
\begin{eqnarray*}
l_n(x_{\sigma(1)},\cdots,x_{\sigma(n-1)},x_{\sigma(n)})=\varepsilon(\sigma)l_n(x_1,\cdots,x_{n-1},x_n),
\end{eqnarray*}
\item[\rm(ii)] {\em (generalized Jacobi Identity)} for all $n\ge 1$,
\begin{eqnarray*}\label{sh-Lie}
\sum_{i=1}^{n}\sum_{\sigma\in \mathbb S_{(i,n-i)} }\varepsilon(\sigma)l_{n-i+1}(l_i(x_{\sigma(1)},\cdots,x_{\sigma(i)}),x_{\sigma(i+1)},\cdots,x_{\sigma(n)})=0.
\end{eqnarray*}
\end{itemize}
\end{defi}

\begin{defi}
 A {\bf  Maurer-Cartan element} of an $L_\infty$-algebra $(\g=\oplus_{k\in\mathbb Z}\g^k,\{l_i\}_{i=1}^{+\infty})$ is an element $\alpha\in \g^0$ satisfying the Maurer-Cartan equation
\begin{eqnarray}\label{MC-equationL}
\sum_{n=1}^{+\infty} \frac{1}{n!}l_n(\alpha,\cdots,\alpha)=0.
\end{eqnarray}
\end{defi}
Let $\alpha$ be a Maurer-Cartan element of an $L_\infty$-algebra $(\g,\{l_i\}_{i=1}^{+\infty})$. For all $k\geq1$ and $x_1,\cdots,x_k\in \g,$
define a series of linear maps $l_k^\alpha:\otimes^k\g\lon\g$ of degree $1$ by
\begin{eqnarray}
 l^{\alpha}_{k}(x_1,\cdots,x_k)=\sum^{+\infty}_{n=0}\frac{1}{n!}l_{n+k}\{\underbrace{\alpha,\cdots,\alpha}_n,x_1,\cdots,x_k\}.
\end{eqnarray}

\begin{thm}{\rm (\cite{Getzler})}\label{thm:twist}
With the above notations, $(\g,\{l^{\alpha}_i\}_{i=1}^{+\infty})$ is an $L_{\infty}$-algebra, obtained from the $L_\infty$-algebra $(\g,\{l_i\}_{i=1}^{+\infty})$ by twisting with the Maurer-Cartan element $\alpha$. Moreover, $\alpha+\alpha'$ is a Maurer-Cartan element of $(\g,\{l_i\}_{i=1}^{+\infty})$ if and only if $\alpha'$ is a Maurer-Cartan element of the twisted $L_{\infty}$-algebra  $(\g,\{l^{\alpha}_i\}_{i=1}^{+\infty})$.
\end{thm}

In \cite{Vo}, Th. Voronov developed the theory of higher derived brackets, which is a useful tool to construct explicit $L_\infty$-algebras.
\begin{defi}{\rm (\cite{Vo})}
A {\bf $V$-data} consists of a quadruple $(L,F,\huaP,\Delta)$, where
\begin{itemize}
\item[$\bullet$] $(L,[\cdot,\cdot])$ is a graded Lie algebra,
\item[$\bullet$] $F$ is an abelian graded Lie subalgebra of $(L,[\cdot,\cdot])$,
\item[$\bullet$] $\huaP:L\lon L$ is a projection, that is $\huaP\circ \huaP=\huaP$, whose image is $F$ and kernel is a  graded Lie subalgebra of $(L,[\cdot,\cdot])$,
\item[$\bullet$] $\Delta$ is an element in $  \ker(\huaP)^1$ such that $[\Delta,\Delta]=0$.
\end{itemize}

\begin{thm}{\rm (\cite{Vo})}\label{thm:db}
Let $(L,F,\huaP,\Delta)$ be a $V$-data. Then $(F,\{{l_k}\}_{k=1}^{+\infty})$ is an $L_\infty$-algebra, where
\begin{eqnarray}\label{V-shla}
l_k(a_1,\cdots,a_k)=\huaP\underbrace{[\cdots[[}_k\Delta,a_1],a_2],\cdots,a_k],\quad\mbox{for homogeneous}~   a_1,\cdots,a_k\in F.
\end{eqnarray}
We call $\{{l_k}\}_{k=1}^{+\infty}$ the {\bf higher derived brackets} of the $V$-data $(L,F,\huaP,\Delta)$. 
\end{thm}

\end{defi}
Let $\g$ be a vector space. We consider the graded vector space $$C^*(\g,\g)=\oplus_{n\ge 0}C^n(\g,\g)=\oplus_{n\ge 0}\Hom (\underbrace{\wedge^{2} \g\otimes \cdots\otimes \wedge^{2}\g}_{n}\wedge \g, \g).$$

\begin{thm}{\rm (\cite{NR bracket of n-Lie})}\label{thm:MCL}
The graded vector space  $C^*(\g,\g)$ equipped with the  graded commutator bracket
\begin{eqnarray}\label{3-Lie-bracket}
[P,Q]_{\Li}=P{\circ}Q-(-1)^{pq}Q{\circ}P,\quad \forall~ P\in C^{p}(\g,\g),Q\in C^{q}(\g,\g),
\end{eqnarray}
is a graded Lie algebra, where $P{\circ}Q\in C^{p+q}(\g,\g)$ is defined by

\begin{small}
\begin{equation*}
\begin{aligned}
&(P{\circ}Q)(\mathfrak{X}_1,\cdots,\mathfrak{X}_{p+q},x)\\
=&\sum_{k=1}^{p}(-1)^{(k-1)q}\sum_{\sigma\in \mathbb S(k-1,q)}(-1)^\sigma P\Big(\mathfrak{X}_{\sigma(1)},\cdots,\mathfrak{X}_{\sigma(k-1)},
Q\big(\mathfrak{X}_{\sigma(k)},\cdots,\mathfrak{X}_{\sigma(k+q-1)},x_{k+q}\big)\wedge y_{k+q},\mathfrak{X}_{k+q+1},\cdots,\mathfrak{X}_{p+q},x\Big)\\
&+\sum_{k=1}^{p}(-1)^{(k-1)q}\sum_{\sigma\in \mathbb S(k-1,q)}(-1)^\sigma P\Big(\mathfrak{X}_{\sigma(1)},\cdots,\mathfrak{X}_{\sigma(k-1)},x_{k+q}\wedge
Q\big(\mathfrak{X}_{\sigma(k)},\cdots,\mathfrak{X}_{\sigma(k+q-1)},y_{k+q}\big),\mathfrak{X}_{k+q+1},\cdots,\mathfrak{X}_{p+q},x\Big)\\
&+\sum_{\sigma\in \mathbb S(p,q)}(-1)^{pq}(-1)^\sigma P\Big(\mathfrak{X}_{\sigma(1)},\cdots,\mathfrak{X}_{\sigma(p)},
Q\big(\mathfrak{X}_{\sigma(p+1)},\cdots,\mathfrak{X}_{\sigma(p+q-1)},\mathfrak{X}_{\sigma(p+q)},x\big)\Big),\\
\end{aligned}
\end{equation*}
\end{small}
 for all $\mathfrak{X}_{i}=x_i\wedge y_i\in \wedge^2 \g$, $~i=1,2,\cdots,p+q$ and $x\in\g.$

  Moreover,  $\mu:\wedge^3\g\longrightarrow\g$ is a $3$-Lie bracket if and only if $[\mu,\mu]_{\Li}=0$, i.e.~$\mu$ is a Maurer-Cartan element of the graded Lie algebra $(C^*(\g,\g),[\cdot,\cdot]_{\Li})$.
  \end{thm}

Let $\rho: \wedge^2\g\rightarrow \gl(\h)$ be an action of a $3$-Lie algebra $(\g,[\cdot,\cdot,\cdot]_{\g})$ on a $3$-Lie algebra $(\h,[\cdot,\cdot,\cdot]_{\h})$. For convenience, we use $\pi:\wedge^3\g\rightarrow\g$ to indicate the $3$-Lie bracket $[\cdot,\cdot,\cdot]_\g$ and $\mu:\wedge^3\h\rightarrow\h$ to indicate the $3$-Lie bracket $[\cdot,\cdot,\cdot]_\h$. In the sequel, we use  $\pi+\rho+\mu$ to denote the element in $\Hom(\wedge^3(\g\oplus \h),\g\oplus\h)$  given by
 \begin{equation}
(\pi+\rho+\mu)(x+u,y+v,z+w)=[x,y,z]_{\g}+\rho(x,y)w+\rho(y,z)u+\rho(z,x)v+[u,v,w]_\h,
\end{equation}
  for all $x,y,z\in\g,~u,v,w\in \h.$ Note that the right hand side is exactly the semidirect product 3-Lie algebra structure given in Proposition  \ref{lem:semi}.
 Therefore by Theorem \ref{thm:MCL}, we have $$[\pi+\rho+\mu,\pi+\rho+\mu]_{\Li}=0.$$

\begin{lem}\label{lem-equation-1}
Let $H:\g\rightarrow\h$ be a crossed homomorphism from a $3$-Lie algebra $(\g,[\cdot,\cdot,\cdot]_\g)$ to a $3$-Lie algebra $(\h,[\cdot,\cdot,\cdot]_\h)$ with respect to an action $\rho$. For all $x,y,z\in \g, u,v,w\in \h,$ we have
\begin{eqnarray*}
&&[[\pi+\rho+\mu,H]_{\Li},H]_{\Li}(x+u,y+v,z+w)\\
&=&2\Big([Hx,Hy,w]_{\h}+[Hx,v,Hz]_{\h}+[u,Hy,Hz]_{\h}\Big)
\end{eqnarray*}
\end{lem}
\begin{proof}
  It follows from straightforward computations.
\end{proof}

\begin{pro}
Let $\rho: \wedge^2\g\rightarrow \gl(\h)$ be an action of a $3$-Lie algebra $(\g,[\cdot,\cdot,\cdot]_{\g})$ on a $3$-Lie algebra $(\h,[\cdot,\cdot,\cdot]_{\h})$. Then we have a $V$-data $(L,F,\huaP,\Delta)$ as follows:
\begin{itemize}
\item[$\bullet$] the graded Lie algebra $(L,[\cdot,\cdot])$ is given by $(C^*(\g\oplus \h,\g\oplus \h),[\cdot,\cdot]_{\Li})$;
\item[$\bullet$] the abelian graded Lie subalgebra $F$ is given by
$$
F=C^*(\g,\h)=\oplus_{n\geq 0}C^{n}(\g,\h)=\oplus_{n\geq 0}\Hom(\underbrace{\wedge^{2} \g\otimes \cdots\otimes \wedge^{2}\g}_{n}\wedge \g, \h);
$$
\item[$\bullet$] $\huaP:L\lon L$ is the projection onto the subspace $F$;
\item[$\bullet$] $\Delta=\pi+\rho+\mu.$
\end{itemize}
Consequently, we obtain an $L_\infty$-algebra $(C^*(\g,\h),l_1,l_3)$, where
\begin{eqnarray*}
l_1(P)&=&[\pi+\rho+\mu,P]_{\Li},\\
l_3(P,Q,R)&=&[[[\pi+\rho+\mu,P]_{\Li},Q]_{\Li},R]_{\Li},
\end{eqnarray*}
for all $P\in C^m(\g,\h),Q\in C^n(\g,\h)$ and $R\in C^k(\g,\h).$
\end{pro}
\begin{proof}
By Theorem \ref{thm:db}, $(F,\{l_k\}^{\infty}_{k=1})$ is an $L_{\infty}$-algebra, where $l_k$ is given by \eqref{V-shla}.
It is obvious that $\Delta=\pi+\rho+\mu \in \ker(\huaP)^{1}$. For all $P\in C^m(\g,\h),Q\in C^n(\g,\h)$ and $R\in C^k(\g,\h)$,
by Lemma \ref{lem-equation-1}, we have
\begin{eqnarray*}
&&[[\pi+\rho+\mu,P]_{\Li},Q]_{\Li}\in\ker(\huaP),
\end{eqnarray*}
which implies that $l_2=0$. Similarly,  we have $l_k=0,$ when $k\geq 4$.
Therefore, the graded vector space $C^*(\g,\h)$ is an $L_{\infty}$-algebra with nontrivial $l_1$, $l_3,$ and other maps are trivial.
\end{proof}

\begin{thm}\label{thm:infty-algebra}
Let $\rho: \wedge^2\g\rightarrow \gl(\h)$ be an action of a $3$-Lie algebra $(\g,[\cdot,\cdot,\cdot]_{\g})$ on a $3$-Lie algebra $(\h,[\cdot,\cdot,\cdot]_{\h})$. Then Maurer-Cartan elements of the $L_{\infty}$-algebra $(C^*(\g,\h),l_1,l_3)$ are precisely crossed homomorphisms from the $3$-Lie algebra $(\g,[\cdot,\cdot,\cdot]_\g)$ to the $3$-Lie algebra $(\h,[\cdot,\cdot,\cdot]_\h)$ with respect to the action $\rho$.
\end{thm}

\begin{proof}
It is straightforward to deduce that
\begin{eqnarray*}
&&[\pi+\rho+\mu,H]_{\Li}(x,y,z)=\rho(y,z)(Hx)+\rho(z,x)(Hy)+\rho(x,y)(Hz)-H\pi(x,y,z);\\
&& [[[\pi+\rho+\mu,H]_{\Li},H]_{\Li},H]_{\Li}(x,y,z)\\
&=&6\mu(Hx,Hy,Hz).
\end{eqnarray*}
Let $H$ be a Maurer-Cartan element of the $L_{\infty}$-algebra $(C^*(\g,\h),l_1,l_3)$. We have
\begin{eqnarray*}
&&\sum_{n=1}^{+\infty} \frac{1}{n!}l_n(H,\cdots,H)(x,y,z)\\
&=&[\pi+\rho+\mu,H]_{\Li}(x,y,z)+\frac{1}{3!}[[[\pi+\rho+\mu,H]_{\Li},H]_{\Li},H]_{\Li}(x,y,z)\\
&=&\mu(Hx,Hy,Hz)+\rho(y,z)(Hx)+\rho(z,x)(Hy)+\rho(x,y)(Hz)-H\pi(x,y,z)\\
&=&0,
\end{eqnarray*}
which implies that $H$ is a crossed homomorphism from the $3$-Lie algebra $(\g,[\cdot,\cdot,\cdot]_\g)$ to the $3$-Lie algebra $(\h,[\cdot,\cdot,\cdot]_\h)$ with respect to the action $\rho$.
\end{proof}

\begin{pro}
Let $H$ be a crossed homomorphism from a $3$-Lie algebra $\g$ to a $3$-Lie algebra $\h$ with respect to an action $\rho$.
 Then $C^*(\g,\h)$ carries a twisted $L_{\infty}$-algebra structure as following:
\begin{eqnarray}
\label{twist-rota-baxter-1}l_1^{H}(P)&=&l_1(P)+\frac{1}{2}l_3(H,H,P),\\
\label{twist-rota-baxter-2}l_2^{H}(P,Q)&=&l_3(H,P,Q),\\
\label{twist-rota-baxter-3}l_3^{H}(P,Q,R)&=&l_3(P,Q,R),\\
l^H_k&=&0,\,\,\,\,k\ge4,
\end{eqnarray}
where $P\in C^m(\g,\h),Q\in C^n(\g,\h)$ and $R\in C^k(\g,\h)$.
\end{pro}
\begin{proof}
 Since $H$ is a Maurer-Cartan element of the $L_{\infty}$-algebra  $(C^*(\g,\h),l_1,l_3)$, by Theorem~\ref{thm:twist}, we have the conclusions.
\end{proof}

The above $L_\infty$-algebra controls deformations of crossed homomorphisms on $3$-Lie algebras.

\begin{thm}\label{thm:deformation}
Let $H:\g\rightarrow\h$ be a crossed homomorphism from a $3$-Lie algebra $(\g,[\cdot,\cdot,\cdot]_\g)$ to a $3$-Lie algebra $(\h,[\cdot,\cdot,\cdot]_\h)$ with respect to an action $\rho$. Then for a linear map $H':\g\rightarrow \h$, $H+H'$ is a crossed homomorphism if and only if $H'$ is a Maurer-Cartan element of the twisted $L_\infty$-algebra $(C^*(\g,\h),l_1^{H},l_2^{H},l_3^{H})$, that is $H'$ satisfies the Maurer-Cartan equation:
$$
l_1^{H}(H')+\frac{1}{2}l_2^{H}(H',H')+\frac{1}{3!}l_3^{H}(H',H',H')=0.
$$
\end{thm}
\begin{proof}
  By Theorem \ref{thm:infty-algebra}, $H+H'$ is a crossed homomorphism if and only if
  $$l_1(H+H')+\frac{1}{3!}l_3(H+H',H+H',H+H')=0.$$
Applying $l_1(H)+\frac{1}{3!}l_3(H,H,H)=0,$ the above condition is equivalent to
  $$l_1(H')+\frac{1}{2}l_3(H,H,H')+\frac{1}{2}l_3(H,H',H')+\frac{1}{6}l_3(H',H',H')=0.$$
That is, $l_1^{H}(H')+\frac{1}{2}l_2^{H}(H',H')+\frac{1}{3!}l_3^{H}(H',H',H')=0,$
\vspace{1mm}which implies that $H'$ is a Maurer-Cartan element of the twisted $L_\infty$-algebra $(C^*(\g,\h),l_1^{H},l_2^{H},l_3^{H})$.
\end{proof}

Next we give the relationship between the coboundary operator ${\rm d}_{\rho_{H}}$
 and the differential $l_1^{H}$ defined by ~\eqref{twist-rota-baxter-1} using the Maurer-Cartan element $H$ of the $L_{\infty}$-algebra  $(C^*(\g,\h),l_1,l_3)$.

 \begin{thm}\label{partial-to-derivation}
 Let $H$ be a crossed homomorphism from a $3$-Lie algebra $(\g,[\cdot,\cdot,\cdot]_\g)$ to a $3$-Lie algebra $(\h,[\cdot,\cdot,\cdot]_\h)$ with respect to an action $\rho$. Then we have
 $$
{\rm d}_{\rho_{H}} f=(-1)^{n-1}l_1^{H} f,\quad \forall f\in \Hom(\underbrace{\wedge^{2} \g\otimes \cdots\otimes \wedge^{2}\g}_{n-1}\wedge\g, \h),~n=1,2,\cdots.
 $$
\end{thm}

\begin{proof}
For all $\mathfrak{X}_i=x_i\wedge y_i\in \wedge^2 \g,~ i=1,2,\cdots,n$ and $x_{n+1}\in \g$, we have
\begin{eqnarray*}
&&l_1(f)(\mathfrak{X}_1,\cdots,\mathfrak{X}_n,x_{n+1})\\
&=&[\pi+\rho+\mu,f]_{\Li}(\mathfrak{X}_1,\cdots,\mathfrak{X}_n,x_{n+1})\\
&=&\Big((\pi+\rho+\mu)\circ f-(-1)^{n-1}f\circ(\pi+\rho+\mu)\Big)(\mathfrak{X}_1,\cdots,\mathfrak{X}_n,x_{n+1})\\
&=&(\pi+\rho+\mu)(f(\mathfrak{X}_1,\cdots,\mathfrak{X}_{n-1},x_n)\wedge y_n,x_{n+1})\\
&&+(\pi+\rho+\mu)(x_n\wedge f(\mathfrak{X}_1,\cdots,\mathfrak{X}_{n-1},y_n), x_{n+1})\\
&&+\sum_{i=1}^{n}(-1)^{n-1}(-1)^{i-1}(\pi+\rho+\mu)(\mathfrak{X}_{i},f(\mathfrak{X}_1,\cdots,\hat{\mathfrak{X}_i},\cdots,\mathfrak{X}_n,x_{n+1}))\\
&&-(-1)^{n-1}\sum_{k=1}^{n-1}\sum_{i=1}^{k}(-1)^{i+1}f\Big(\mathfrak{X}_1,\cdots,\hat{\mathfrak{X}}_{i},\cdots,\mathfrak{X}_{k}, (\pi+\rho+\mu)(\mathfrak{X}_{i},x_{k+1})\wedge y_{k+1},\mathfrak{X}_{k+2},\cdots,\mathfrak{X}_{n},x_{n+1}\Big)\\
&&-(-1)^{n-1}\sum_{k=1}^{n-1}\sum_{i=1}^{k}(-1)^{i+1}f\Big(\mathfrak{X}_1,\cdots,\hat{\mathfrak{X}}_{i},\cdots,\mathfrak{X}_{k},x_{k+1}\wedge (\pi+\rho+\mu)(\mathfrak{X}_{i},y_{k+1}),\mathfrak{X}_{k+2},\cdots,\mathfrak{X}_{n},x_{n+1}\Big)\\
&&-(-1)^{n-1}\sum_{i=1}^{n}(-1)^{i+1}f\Big(\mathfrak{X}_1,\cdots,\hat{\mathfrak{X}}_{i},\cdots,\mathfrak{X}_{n},(\pi+\rho+\mu)(\mathfrak{X}_{i},x_{n+1})\Big)\\
&=&\rho(y_n,x_{n+1})f(\mathfrak{X}_{1},\cdots,\mathfrak{X}_{n-1},x_n)+\rho(x_{n+1},x_n)f(\mathfrak{X}_{1},\cdots,\mathfrak{X}_{n-1},y_n)\\
&&+\sum_{i=1}^{n}(-1)^{n-1}(-1)^{i-1}\rho(x_i,y_i)f(,\mathfrak{X}_{1},\cdots,\hat{,\mathfrak{X}_{i}},\cdots,,\mathfrak{X}_{n},x_{n+1})\\
&&-(-1)^{n-1}\sum_{k=1}^{n-1}\sum_{i=1}^{k}(-1)^{i+1}f\Big(\mathfrak{X}_1,\cdots,\hat{\mathfrak{X}}_{i},\cdots,\mathfrak{X}_{k}, \pi(x_i,y_i,x_{k+1})\wedge y_{k+1},\mathfrak{X}_{k+2},\cdots,\mathfrak{X}_{n},x_{n+1}\Big)\\
&&-(-1)^{n-1}\sum_{k=1}^{n-1}\sum_{i=1}^{k}(-1)^{i+1}f\Big(\mathfrak{X}_1,\cdots,\hat{\mathfrak{X}}_{i},\cdots,\mathfrak{X}_{k},x_{k+1}\wedge \pi(x_i,y_i,y_{k+1}),\mathfrak{X}_{k+2},\cdots,\mathfrak{X}_{n},x_{n+1}\Big)\\
&&-(-1)^{n-1}\sum_{i=1}^{n}(-1)^{i+1}f\Big(\mathfrak{X}_1,\cdots,\hat{\mathfrak{X}}_{i},\cdots,\mathfrak{X}_{n},\pi(x_i,y_i,x_{n+1})\Big).
\end{eqnarray*}
By Lemma \ref{lem-equation-1}, we have
\begin{eqnarray*}
&&l_3(H,H,f)(\mathfrak{X}_1,\cdots,\mathfrak{X}_n,x_{n+1})\\
&=&[[[\pi+\rho+\mu,H]_{\Li},H]_{\Li},f]_{\Li}(\mathfrak{X}_1,\cdots,\mathfrak{X}_n,x_{n+1})\\
&=&[[\pi+\rho+\mu,H]_{\Li},H]_{\Li}\Big(f(\mathfrak{X}_1,\cdots,\mathfrak{X}_{n-1},x_n)\wedge y_n,x_{n+1}\Big)\\
&&+[[\pi+\rho+\mu,H]_{\Li},H]_{\Li}\Big(x_n\wedge f(\mathfrak{X}_1,\cdots,\mathfrak{X}_{n-1},y_n),x_{n+1}\Big)\\
&&+\sum_{i=1}^{n}(-1)^{n-1}(-1)^{i-1}[[\pi+\rho+\mu,H]_{\Li},H]_{\Li}\Big(\mathfrak{X}_i, f(\mathfrak{X}_1,\cdots,\hat{\mathfrak{X}}_{i},\cdots,\mathfrak{X}_{n},x_{n+1})\Big)\\
&&-(-1)^{n-1}\sum_{k=1}^{n-1}\sum_{i=1}^{k}(-1)^{i+1}f\Big(\mathfrak{X}_1,\cdots,\hat{\mathfrak{X}}_{i},\cdots,\mathfrak{X}_{k},[[\pi+\rho+\mu,H]_{\Li},H]_{\Li}(\mathfrak{X}_{i},x_{k+1})\wedge y_{k+1}\\
&&+x_{k+1}\wedge [[\pi+\rho+\mu,H]_{\Li},H]_{\Li}(\mathfrak{X}_{i},y_{k+1}),\mathfrak{X}_{k+2},\cdots,\mathfrak{X}_{n},x_{n+1}\Big)\\
&&-(-1)^{n-1}\sum_{i=1}^{n}(-1)^{i+1}f\Big(\mathfrak{X}_1,\cdots,\hat{\mathfrak{X}}_{i},\cdots,\mathfrak{X}_{n},[[\pi+\rho+\mu,H]_{\Li},H]_{\Li}(\mathfrak{X}_{i},x_{n+1})\Big)\\
&=&2\Big([f(\mathfrak{X}_1,\cdots,\mathfrak{X}_{n-1},x_n),Hy_n,Hx_{n+1}]_{\h}+[Hx_n,f(\mathfrak{X}_1,\cdots,\mathfrak{X}_{n-1},y_n),Hx_{n+1}]_{\h}\\
&&+\sum_{i=1}^{n}(-1)^{n-1}(-1)^{i-1}[Hx_i,Hy_i,f(\mathfrak{X}_1,\cdots,\hat{\mathfrak{X}_{i}},\mathfrak{X}_{n},x_{n+1})]_{\h}\Big)
\end{eqnarray*}
Thus, we deduce that ${\rm d}_{H} f=(-1)^{n-1}\Big(l_1(f)+\frac{1}{2}l_3(H,H,f)\Big)$, that is  ${\rm d}_{H} f=(-1)^{n-1}l_1^{H} f$.
\end{proof}


 \end{document}